\documentclass[12pt]{amsart}
\usepackage{amsmath,amssymb,amsbsy,amsfonts,latexsym,amsopn,amstext,
                                               amsxtra,euscript,amscd,color,bm}
                                                                                          
\usepackage[colorlinks,linkcolor=blue,anchorcolor=blue,citecolor=blue]{hyperref}

\newfont{\teneufm}{eufm10}
\newfont{\seveneufm}{eufm7}
\newfont{\fiveeufm}{eufm5}
\newfam\eufmfam
                \textfont\eufmfam=\teneufm \scriptfont\eufmfam=\seveneufm
                \scriptscriptfont\eufmfam=\fiveeufm
\def\bbbc{{\mathchoice {\setbox0=\hbox{$\displaystyle\rm C$}\hbox{\hbox
to0pt{\kern0.4\wd0\vrule height0.9\ht0\hss}\box0}}
{\setbox0=\hbox{$\textstyle\rm C$}\hbox{\hbox
to0pt{\kern0.4\wd0\vrule height0.9\ht0\hss}\box0}}
{\setbox0=\hbox{$\scriptstyle\rm C$}\hbox{\hbox
to0pt{\kern0.4\wd0\vrule height0.9\ht0\hss}\box0}}
{\setbox0=\hbox{$\scriptscriptstyle\rm C$}\hbox{\hbox
to0pt{\kern0.4\wd0\vrule height0.9\ht0\hss}\box0}}}}
\def\bbbq{{\mathchoice {\setbox0=\hbox{$\displaystyle\rm
Q$}\hbox{\raise 0.15\ht0\hbox to0pt{\kern0.4\wd0\vrule
height0.8\ht0\hss}\box0}} {\setbox0=\hbox{$\textstyle\rm
Q$}\hbox{\raise 0.15\ht0\hbox to0pt{\kern0.4\wd0\vrule
height0.8\ht0\hss}\box0}} {\setbox0=\hbox{$\scriptstyle\rm
Q$}\hbox{\raise 0.15\ht0\hbox to0pt{\kern0.4\wd0\vrule
height0.7\ht0\hss}\box0}} {\setbox0=\hbox{$\scriptscriptstyle\rm
Q$}\hbox{\raise 0.15\ht0\hbox to0pt{\kern0.4\wd0\vrule
height0.7\ht0\hss}\box0}}}}
\def\bbbt{{\mathchoice {\setbox0=\hbox{$\displaystyle\rm
T$}\hbox{\hbox to0pt{\kern0.3\wd0\vrule height0.9\ht0\hss}\box0}}
{\setbox0=\hbox{$\textstyle\rm T$}\hbox{\hbox
to0pt{\kern0.3\wd0\vrule height0.9\ht0\hss}\box0}}
{\setbox0=\hbox{$\scriptstyle\rm T$}\hbox{\hbox
to0pt{\kern0.3\wd0\vrule height0.9\ht0\hss}\box0}}
{\setbox0=\hbox{$\scriptscriptstyle\rm T$}\hbox{\hbox
to0pt{\kern0.3\wd0\vrule height0.9\ht0\hss}\box0}}}}
\def\bbbs{{\mathchoice
{\setbox0=\hbox{$\displaystyle     \rm S$}\hbox{\raise0.5\ht0\hbox
to0pt{\kern0.35\wd0\vrule height0.45\ht0\hss}\hbox
to0pt{\kern0.55\wd0\vrule height0.5\ht0\hss}\box0}}
{\setbox0=\hbox{$\textstyle        \rm S$}\hbox{\raise0.5\ht0\hbox
to0pt{\kern0.35\wd0\vrule height0.45\ht0\hss}\hbox
to0pt{\kern0.55\wd0\vrule height0.5\ht0\hss}\box0}}
{\setbox0=\hbox{$\scriptstyle      \rm S$}\hbox{\raise0.5\ht0\hbox
to0pt{\kern0.35\wd0\vrule height0.45\ht0\hss}\raise0.05\ht0\hbox
to0pt{\kern0.5\wd0\vrule height0.45\ht0\hss}\box0}}
{\setbox0=\hbox{$\scriptscriptstyle\rm S$}\hbox{\raise0.5\ht0\hbox
to0pt{\kern0.4\wd0\vrule height0.45\ht0\hss}\raise0.05\ht0\hbox
to0pt{\kern0.55\wd0\vrule height0.45\ht0\hss}\box0}}}}
\def\bbbz{{\mathchoice {\hbox{$\sf\textstyle Z\kern-0.4em Z$}}
{\hbox{$\sf\textstyle Z\kern-0.4em Z$}} {\hbox{$\sf\scriptstyle
Z\kern-0.3em Z$}} {\hbox{$\sf\scriptscriptstyle Z\kern-0.2em
Z$}}}}

\newtheorem{theorem}{Theorem}
\newtheorem{lemma}[theorem]{Lemma}

\newtheorem{cor}[theorem]{Corollary}

\newtheorem{quest}[theorem]{Open Question}
\newtheorem{rem}[theorem]{Remark}

\numberwithin{table}{section}
\numberwithin{equation}{section}
\numberwithin{figure}{section}
\numberwithin{theorem}{section}

\def\squareforqed{\hbox{\rlap{$\sqcap$}$\sqcup$}}
\def\qed{\ifmmode\squareforqed\else{\unskip\nobreak\hfil
\penalty50\hskip1em\null\nobreak\hfil\squareforqed
\parfillskip=0pt\finalhyphendemerits=0\endgraf}\fi}

\def\cA{{\mathcal A}}

\def\cE{{\mathcal E}}

\def\cI{{\mathcal I}}
\def\cJ{{\mathcal J}}

\def\cL{{\mathcal L}}

\def\cP{{\mathcal P}}
\def\cQ{{\mathcal Q}}
\def\cR{{\mathcal R}}
\def\cS{{\mathcal S}}

\def\cU{{\mathcal U}}
\def\cV{{\mathcal V}}

\def \sf {\mathfrak s}

\newcommand{\ignore}[1]{}

\def\e{\mathbf{e}}

\def \Z{\mathbb{Z}}

\def \Z{\mathbb{Z}}

\def\mand{\qquad\mbox{and}\qquad}

\def\\{\cr}
\def\({\left(}
\def\){\right)}
\def\fl#1{\left\lfloor#1\right\rfloor}
\def\rf#1{\left\lceil#1\right\rceil}

\def\em{\mathbf{e}_m}
\def\e{\mathbf{e}}

\begin{document}

\title{Cyclotomic Coefficients: Gaps and Jumps}

\author[Camburu]{Oana-Maria Camburu}
\address{Department of Computer Science, University of Oxford, Oxford OX1 3QD, United Kingdom }
\email{oana-maria.camburu@cs.ox.ac.uk}

 \author[Ciolan]{Emil-Alexandru Ciolan}

\address{Rheinische Friedrich-Wilhelms-Universit\"at Bonn, Regina-Pacis-Weg 3, D-53113 Bonn, Germany}
 \email{ciolan@uni-bonn.de}

\author[Luca]{Florian Luca} 
\address{School of Mathematics, University of the Witwatersrand, 
Private Bag X3, Wits 2050, South Africa}
\email{florian.luca@wits.ac.za}

 \author[Moree] {Pieter Moree}

\address{Max-Planck-Institut f{\"u}r Mathematik, Vivatsgasse 7, D-53111 Bonn, Germany}
\email{moree@mpim-bonn.mpg.de}

 \author[Shparlinski] {Igor E. Shparlinski}

\address{Department of Pure Mathematics, University of New South Wales,
Sydney, NSW 2052, Australia}
\email{igor.shparlinski@unsw.edu.au}

%
%
%
%
%
%
%
%
%

\begin{abstract} We improve several recent
results by Hong, Lee, Lee and Park (2012) on gaps and  Bzd{\c e}ga (2014)   
on jumps amongst the coefficients of 
cyclotomic polynomials. Besides direct improvements, we also introduce 
several new techniques that have never been used in this area. 
\end{abstract}


\maketitle

\section{Introduction}

As usual, for an integer $n\ge1$, we use $\Phi_n(Z)$ to denote 
the {\it $n$th cyclotomic polynomial\/}, that is,
$$
\Phi_n(Z) = \prod_{\substack{j=0\\\gcd(j,n)=1}}^{n-1}
\(Z - \e_n(j)\), 
$$
where for an integer $m\ge 1$ and a real $z$, we put
$$\em(z) = \exp(2 \pi i z/m).
$$
Clearly $\deg \Phi_n= \varphi(n)$, where $\varphi(n)$ is the Euler 
function. 
Using the above definition one sees that
\begin{equation}
\label{allebegin...} 
Z^n-1=\prod_{d\mid n}\Phi_d(Z).
\end{equation} 
The M\"obius inversion formula then yields
\begin{equation}
\label{prod1}
\Phi_n(Z)=\prod_{d\mid n}(Z^d-1)^{\mu(n/d)},
\end{equation}
where $\mu(n)$ denotes the M\"obius function. 

We  write 
$$
\Phi_n(Z)=\sum_{k=0}^{\varphi(n)}a_n(k)Z^k.
$$
For $n>1$ clearly $Z^{\varphi(n)}\Phi_n(1/Z)=\Phi_n(Z)$ and so 
\begin{equation}
\label{eq:sym}
a_n(k)=a_n(\varphi(n)-k), \qquad 0\le k\le \varphi(n),\qquad n>1.
\end{equation}
Recently, there has been a burst of activity in studying 
the cyclotomic coefficients $a_n(k)$, 
see, for example,~\cite{Bzd-0,Bzd-1,Bzd-2,Bzd-3,CGMZ,Fouv,GaMo1,GaMo2,GaMoWi,MoRo,ZhZh}
and references therein. Furthermore, in several works 
{\it inverse cyclotomic polynomials\/} 
$$
\Psi_n(Z) = (Z^n-1)/\Phi_n(Z)
$$
have also been considered, 
see~\cite{Bzd-4,HLL, HLLP, HLLP2, Mor1}.

The identities $\Phi_{2n}(Z)=\Phi_n(-Z)$, with $n>1$ odd and 
$\Phi_{pm}(Z)=\Phi_m(Z^p)$ if $p\mid m$, show that, as far as the study of coefficients is concerned,
the complexity of $\Phi_n(Z)$ is determined by its number of distinct odd prime factors.
Most of the recent activity concerns the so called {\it binary\/} and {\it ternary\/} 
cyclotomic polynomials, which are polynomials $\Phi_n(Z)$ 
with $n=pq$ and $n=pqr$, respectively, where $p$, $q$ and $r$ are 
pairwise distinct odd primes. 
In particular, the long standing
{\it Beiter conjecture\/} about  coefficients of ternary 
cyclotomic polynomials has been shown to be wrong
(in a very strong sense) by Gallot and Moree~\cite{GaMo2}, 
see also~\cite{Bzd-0}. 
It is quite remarkable that the results of~\cite{GaMo2} are
based on seemingly foreign to the problem and deep analytic 
results such as bounds 
of Kloosterman sums (see~\cite[Theorem~11.11]{IwKow}) 
and a result of  Duke,  Friedlander
and Iwaniec~\cite{DuFrIw1} on the distribution of roots of
quadratic congruences.  

Furthermore, Fouvry~\cite{Fouv} has used bounds of 
exponential sums with reciprocals of primes from~\cite{FouShp} 
in studying the sparsity of binary cyclotomic polynomials
and improved a result of Bzd{\c e}ga~\cite{Bzd-1}. 
It is quite possible that more recent bounds of Baker~\cite{Bak}
and Irving~\cite{Irv} can lead to further progress in this
direction.

Here we continue to study ternary cyclotomic and inverse 
cyclotomic polynomials and using  some classical and 
more recent tools from analytic number theory, we 
improve several previous results. We also show how to employ a rather nonstandard tool of using numerical semigroups to study the binary cyclotomic polynomials. We give a brief summary of our
contributions in Section~\ref{recap}.

We recall that the notations $U = O(V)$, $U \ll V$ 
 and $V \gg U$ 
are all equivalent to the assertion that the
inequality $|U|\le cV$ holds for some constant $c>0$. If
the constant $c$ depends on a parameter, say $\varepsilon$, we
indicate this as $U \ll_{\varepsilon} V$, etc.
We also write $U \asymp V$ if  $U \ll V \ll U$.
For an integer $m$ and a real $M\ge 1$, we write 
$m \sim M$ to indicate that $m \in [M, 2M]$. If $\cA$ is 
a set of non-negative integers we denote by $\cA(x)$ the subset of integers
$n\in \cA$ with $n\le x$.

As usual, we let $\pi(x)$ denote the number of primes $p \le x$. 
A few times we  use the estimate 
\begin{equation}
\label{simplepi}
\pi(x)=\frac{x}{\log x}+O\(\frac{x}{(\log x)^2}\).
\end{equation}

We always use the letters $\ell$, $p$, $q$ and $r$ to denote prime
numbers, while $k$, $m$ and $n$ always denote positive integers.

\section{Maximum gap}

\subsection{Definitions and background}
\label{achtergrond}
Given a polynomial
$$f(Z) = c_1Z^{e_1} + \cdots  +c_t Z^{e_t}\in  {\Z} [Z],\text{with~} c_i \ne 0
\text{~and~}e_1< \cdots < e_t,
$$ 
we define the {\it maximum gap\/} of $f$ as
$$
g(f) = \max_{1\le i<t}(e_{i+1}-e_i),
$$
where we set $g(f) =0$ when $t = 1$. 
The study  of $g(\Phi_n)$ and $g(\Psi_n)$ has been initiated by 
Hong, Lee, Lee and Park~\cite{HLLP} 
who have reduced the study of these gaps to the case where $n$ is square-free and odd. 
They were led to their gap study in an attempt to provide a simple and exact formula for the minimum
Miller loop length in the Ate$_i$ pairing arising in elliptic curve cryptography, 
see~\cite{HLL,HLLP,HLLP2}.
As they write in~\cite{HLLP2}, the crucial idea in finding this exact formula is that it
becomes more manageable when it is suitably recast in terms of inverse cyclotomic polynomials and consequently
turns into a problem involving the maximum gaps.

It is easy to see~\cite{HLLP}  
that if $p<q$ are odd primes, then 
$$
g(\Phi_{p})=1, \qquad g(\Psi_{p})=1, \qquad g(\Psi_{pq})=q-p+1.
$$ 
The simplest non-trivial case occurs when $n$ is a product 
of two distinct primes, where by~\cite[Theorem~1]{HLLP}, 
for two primes $3 \le p< q$ we have 
\begin{equation}
\label{eq:pq}
g(\Phi_{pq})=p-1.
\end{equation}

For ternary $n$ we have by~\cite[Theorems~2 and~3]{HLLP} the following
partial result on $g(\Psi_n)$.
\begin{lemma} 
 \label{lem:HLLP-3}
Put 
$$
\cR_3=   \{n =  pqr  ~:~p<q<r\ \text{primes}, \ 4(p -1) > q,\ p^2>r\}.
$$
Let $n=pqr$ with primes $2< p<q<r$. Then:
\begin{itemize}
\item we have
\begin{equation*}
\max\left\{p-1,\frac{2n}{p}-\deg \Psi_n\right\}\le g(\Psi_n)<2n \(\frac{1}{p}+\frac{1}{q}+\frac{1}{r}\)-\deg (\Psi_n);
\end{equation*}
\item if $n \not \in \cR_3$, then $g(\Psi_n)=2n/p-\deg \Psi_n$.
\end{itemize}
\end{lemma}
Note that $\cR_3$ consists of ternary integers only.
Although in this work we are mostly interested in the set $\cR_3$ that appears in
Lemma~\ref{lem:HLLP-3}, we first make some comments concerning~\eqref{eq:pq}.

Moree~\cite{Mor2} derives~\eqref{eq:pq} using a very different 
technique which is based on numerical semigroups. We come back to this in 
Section~\ref{numericalsemi}.

Here we present yet another short proof of~\eqref{eq:pq} communicated to us 
by Nathan Kaplan. We start by noticing that the nonzero coefficients of
$\Phi_{pq}(Z)$ alternate between $1$ and $-1$, see~\cite{Carlitz,Mor2}.
From~\eqref{prod1} one easily obtains that $\Phi_p(Z)\Phi_{pq}(Z) =\Phi_p(Z^q)$.  Suppose there is a gap of 
length $p$ or greater, say, for some positive integers $b \ge a+p$ and there are no nonzero coefficients between $Z^a$ and $Z^b$.  One of these
coefficients is $1$ and the other is $-1$. Now consider the product
$\Phi_p(Z)\Phi_{pq}(Z)$ and examine  at the 
coefficients of the terms with
$Z^{a+p-1}$ and $Z^b$. 
One of these coefficients is $1$ and the other is $-1$, contradicting the fact
that $\Phi_p(Z^q)$ has no coefficient equal to $-1$ and~\eqref{eq:pq}
follows. 

\subsection{Main result}
We now estimate $\cR_3(x)$.
We frequently make use of the bound
\begin{equation}
\label{eq:k}
\sum_{p<z} p^k\ll \frac{z^{k+1}}{\log z}, 
\end{equation}
which holds for any fixed $k\ge 1$ and real $z \ge 2$  and 
follows easily from~\eqref{simplepi} by partial summation. 
\begin{theorem}
\label{florian}
We have
$$
\#\cR_3(x)=\frac{c x}{(\log x)^2}+O\(\frac{x\log\log x}{(\log x)^3}\),
$$
where $c=(1+\log 4) \log 4=3.30811\ldots.$ 
\end{theorem}

\begin{proof}
We take $\cS(x)=\{n:n\le x/(\log x)^3\}$. Clearly, 
$\# \cS(x)\le x(\log x)^{-3}$.
So, from now on, we only consider 
$$
n\in \cL_3(x) =  \cR_3(x)\setminus \cS(x).
$$ 

Now, for  $n =pqr \in \cL_3(x)$ we have 
$$
p^3<pqr\le x \mand 4p^4=p^2 (4p) p>rqp=n>\frac{x}{(\log x)^3}.$$
Thus, 
$$
p\in \cI(x)=[y(x),x^{1/3}],
$$
where 
$$
y(x) = \frac{x^{1/4}}{\sqrt{2}(\log x)^{3/4}}.
$$
In particular, $\log p\asymp \log q\asymp \log r\asymp \log x$. 

We now fix a prime $p \in \cI(x)$ and consider the interval 
$$
\cJ_p=(p,4(p-1)).
$$  
If  $n =pqr \in \cL_3(x)$, 
then $q\in \cJ_p$.  For fixed $p$ and $q\in \cJ_p$, we have 
$r<\min\{p^2,x/pq\}$.  We distinguish the following two cases depending of 
whether $\min\{p^2,x/pq\} = p^2$, and denote by $\cU_{3}(x)$ the set of such $n \in \cL_3(x)$, 
or  $\min\{p^2,x/pq\} = x/pq$,  and denote by $\cV_{3}(x)$ the set of such $n \in \cL_3(x)$.

We estimate the cardinalities of the sets $\cU_{3}(x)$  and $\cV_{3}(x)$ separately.

First, assume that $n = pqr \in \cU_3(x)$. Therefore $p^4<p^2(pq)<x$, so in fact $p<x^{1/4}$. For fixed $p$ and $q$, the number of primes $r<p^2$ is $O(\pi(p^2))=O(p^2/\log x)$. 
Further, the number of  choices of $q<4p$ is $O(\pi(4p))=O(p/\log x)$. Thus, for fixed $p\in  \cI(x)$, the number of choices for the 
pairs $(q,r)$ is 
$$
O\( \frac{p^2}{\log x} \times \frac{p}{\log x}\)=O\(\frac{p^3}{(\log x)^2}\).
$$ 

Summing up the above bound over all the choices $p\le x^{1/4}$ and using~\eqref{eq:k} with $k=3$ we derive
\begin{equation}
\label{eq:U3}
\# \cU_{3}(x) \ll \frac{1}{(\log x)^2} \sum_{p<x^{1/4}} p^3 \ll \frac{x}{(\log x)^3},
\end{equation}
which gets absorbed in the error term of the desired asymptotic formula for $\cR_3(x)$.  

So we now consider $n = pqr \in \cV_3(x)$. In this case,
 using~\eqref{simplepi}
and $\log (x/pq)>\log q \asymp \log x$, we see that  the number of choices for  the prime $r\in (q, x/pq)$ when $p$ and $q$ are fixed is 
\begin{equation}
\label{eq:2}
\pi\(\frac{x}{pq}\)-\pi(q)=\frac{x}{pq\log(x/pq)}+O\(\frac{x}{pq(\log x)^2}+\frac{q}{\log q}\).
\end{equation}

Note that 
$$
\log(x/pq)=\log x-\log p-\log q=\log x-2\log p+O(1)
$$
for  $q\in \cJ_p$. Thus, using  $(1+z)^{-1}=1+O(z)$ and 
$\log (x/p^2)\asymp \log x$, a simple calculation shows that 
$$
\frac{1}{\log(x/pq)}  = \frac{1}{\log(x/p^2)}+O\(\frac{1}{(\log x)^2}\). 
$$
Hence, using~\eqref{eq:2}, we see that the number of choices for $r$ when $p$ and $q$ are fixed is
\begin{equation}
\label{eq:3}
\frac{x}{pq\log(x/p^2)}+O\(\frac{x}{pq (\log x)^2}+\frac{q}{\log x}\).
\end{equation}
Now we sum up over $q\in (p,4(p-1))$ and use the Mertens formula
 (see~\cite[Equation~(2.15)]{IwKow})  in the form 
\begin{equation}
\label{eq:4}
\sum_{\substack{\ell<X\\\ell~\text{prime}}} \frac{1}{\ell} =\log\log X+\alpha+O\(\frac{1}{(\log X)^2}\), 
\qquad X \to \infty, 
\end{equation}
with some constant $\alpha$ to deduce that
\begin{equation}
\label{eq:Mert}
\begin{split}
\sum_{p<q<4(p-1)}  \frac{1}{q}   &=  \log\log 4p-\log\log p+O\(\frac{1}{(\log p)^2}\) \\
 & =  \frac{\log 4}{\log p} +O\(\frac{1}{(\log x)^2}\).
 \end{split}
\end{equation}
Summing over the choices $q\in \cJ_p$, we see that the main term in~\eqref{eq:3}
contributes 
$$
\frac{\log 4}{\log p \log(x/p^2)}+O\(\frac{x}{p(\log x)^3}\).
$$
We now sum up the first error terms in~\eqref{eq:3}, getting 
$$
\sum_{p<q<4(p-1)} \frac{x}{pq(\log x)^2}  \ll \frac{x}{p(\log x)^2} \sum_{p<q<4p}\frac{1}{q} \ll \frac{x}{p(\log x)^3},
$$
where we have used~\eqref{eq:Mert}, and the second error terms in~\eqref{eq:3} getting
$$
\sum_{p<q<4(p-1)} \frac{q}{\log x} \ll \frac{1}{\log x} \sum_{q<4p} q \ll \frac{p^2}{(\log x)^2},
$$
where we have used~\eqref{eq:k} with $k=1$ and $z=4p$. 

Hence, the number of choices for the pair $(q,r)$ when $p\in \cI(x)$ is
$$
\frac{x \log 4 }{p \log p \log(x/p^2)} +O\(\frac{x}{p(\log x)^3}+\frac{p^2}{(\log x)^2}\).
$$
Now we sum over the primes $p\in \cI(x)$. Since $y(x)>x^{1/5}$ for large $x$, for the total error term we 
obtain 
\begin{equation*}
\begin{split}
\sum_{p\in \cI(x)}  \(\frac{x}{p(\log x)^3}+\frac{p^2}{(\log x)^2}\)
& \le \frac{x}{(\log x)^3} \sum_{x^{1/5}<p<x}\frac{1}{p}+
 \frac{1}{(\log x)^2} \sum_{p<x^{1/3}} p^2\\& \ll \frac{x}{(\log x)^3}
 \end{split}
\end{equation*}
by~\eqref{eq:4} (applied to the sum of $1/p$) and ~\eqref{eq:k}
(applied to the sum of $p^2$). Thus, 
$$
\# \cV_3(x) =  \sum_{p\in \cI(x)} \frac{x \log 4}{p\log p \log(x/p^2)} +O\(\frac{x}{(\log x)^3}\).
$$

For the sum of the main term (after we pull out $x \log 4$), using the  Stieltjes integral and then partial integration, we obtain 
\begin{equation*}
\begin{split}
\sum_{p\in \cI(x)}  \frac{1}{p\log p\log(x/p^2)}&=
\int_{y(x)}^{x^{1/3}} \frac{ d\pi (t)}{t \log t \log(x/t^2)}
\\& =   \int_{y(x)}^{x^{1/3}} \frac{d \pi(t)}{t \log t\log(x/t^2)}=\frac{\pi(t)}{t \log t\log(x/t^2)}\Bigr|_{t=y(x)}^{t=x^{1/3}} \\&-\int_{y(x)}^{x} \pi(t) \frac{d}{dt}  
\(\frac{1}{t\log t\log(x/t^2)}\).
 \end{split}
\end{equation*}
Since $\pi (t) /t=O(1/\log t)$, we see that the first term above is 
$$
\frac{\pi(t)}{t\log t\log(x/t^2)}\Bigr|_{t=y(x)}^{t=x^{1/3}}=O\(\frac{1}{(\log x)^3}\).
$$
We have
\begin{equation*}
\begin{split}
\frac{d}{dt} & \(\frac{1}{t\log t\log(x/t^2)}\) =   - \frac{\log t \log (x/t^2)+\log (x/t^2)-2\log t}{t^2 (\log t)^2 (\log x-2\log t)^2}\\
& \qquad \qquad =  -\frac{1}{t^2 \log t (\log x-2\log t)}+O\(\frac{1}{t^2 (\log x)^3}\)\\
&\qquad \qquad  =  -\frac{1}{t(\log t)^2(\log x-2\log t)}+O\(\frac{1}{t\log t(\log x)^3}\).
 \end{split}
\end{equation*}
Thus, using~\eqref{simplepi}, we have 
\begin{equation*}
\begin{split}
\int_{y(x)}^{x} \pi(t)  \frac{d}{dt}  \(\frac{1}{t\log t \log(x/t^2)}\) dt 
 = & -  \int_{y(x)}^{x^{1/3}} \frac{dt}{t(\log t)^2( \log x-2\log t)} 
\\&+  O\(\int_{y(x)}^{x^{1/3}} \frac{dt}{t(\log t) (\log x)^3}\).
 \end{split}
\end{equation*}
Since 
$$
 \int_{y(x)}^{x^{1/3}} \frac{dt}{t\log t} = \log\log t\Bigr|_{t=y(x)}^{t=x^{1/3}} \ll 1,
 $$
we derive
\begin{equation}
\begin{split}
\label{eq:5}
\#\cV_3(x)=  x \log 4 \int_{y(x)}^{x^{1/3}} \frac{dt}{t(\log t)^2( \log x-2\log t)}
+O\(\frac{x}{(\log x)^3}\).
 \end{split}
\end{equation}
Inside the integral, we make the change of variable $t=x^ u$, for which $dt= x^u \log x du$. Further, 
we have $\log t=u\log x$ and we also have $\log x-2\log t=(\log x)(1-2u)$. Defining 
$$
z(x) =\frac{\log y(x)}{\log x}=\frac{1}{4}-\frac{\log(2^{1/2} (\log x)^{3/4})}{\log x}=\frac{1}{4}+O\(\frac{\log\log x}{\log x}\),
$$
we see  that
\begin{equation}
\label{eq:6}
 \int_{y(x)}^{x^{1/3}} \frac{dt}{t(\log t)^2( \log x-2\log t)} =  \frac{1}{(\log x)^2} \int_{z}^{1/3} \frac{du}{u^2(1-2u)}.
\end{equation}
We now write
\begin{equation}
\begin{split}
\label{eq:7}
 \int_{z}^{1/3} \frac{du}{u^2(1-2u)} & = \int_{1/4}^{1/3} \frac{du}{u^2(1-2u)} +\int_{z}^{1/4} \frac{du}{u^2(1-2u)} \\
 & =    \int_{1/4}^{1/3} \frac{du}{u^2(1-2u)} +O\(\frac{\log\log x}{\log x}\).
 \end{split}
\end{equation}
 One verifies that 
 $$
 \frac{d}{du}\(- 2 \log(\frac{1}{u} -2)  - \frac{1}{u}\) =  \frac{1}{u^2(1-2u)}.
 $$
 Hence,
\begin{equation}
\label{eq:8}
  \int_{1/4}^{1/3} \frac{du}{u^2(1-2u)} 
  = - \(2 \log(\frac{1}{u} -2)  + \frac{1}{u}\)\Bigr|_{u=1/3}^{u=1/4} = 1 + \log 4. 
\end{equation}
Substituting~\eqref{eq:7} and~\eqref{eq:8} in~\eqref{eq:6} and recalling~\eqref{eq:5}
we  obtain 
 \begin{equation}
\label{eq:V3}
\#\cV_3(x)=  \frac{x (1+\log 4)\log 4}{(\log x)^2}
+O\(\frac{x\log\log x}{(\log x)^3}\).
\end{equation}

Now, combining~\eqref{eq:U3} and~\eqref{eq:V3}, we 
conclude the proof.
\end{proof}

Let 
$$
\cQ_3(x) =   \{n \le x~:~n = pqr \ \text{for primes}\ 2<p<q<r\}.
$$
Using~\cite[Theorem~4, Section~II.6.1]{Ten} we conclude that
\begin{equation}
\label{asgelijk} 
\# \cQ_3(x) = (1+ o(1)) \frac{x (\log \log x)^2}{2\log x}
\end{equation}
as $x \to \infty$. Combining this estimate with Theorem~\ref{florian}, 
we immediately derive the following comparison of 
$\# \cR_3(x)$ and $\# \cQ_3(x)$.
\begin{cor}
\label{precisie}
As $x \to \infty$ we have 
$$\# \cR_3(x)  =   \frac{2(1+\log 4) \log 4 + o(1)}{(\log x) (\log \log x)^2 } \# \cQ_3(x).$$
\end{cor}

Corollary~\ref{precisie} justifies the claim in~\cite[Remark~1]{HLLP} 
that $\cR_3(x)$ is a sparse set, that is, that 
$\# \cR_3(x) = o(\# \cQ_3(x))$ as $x \to \infty$.
Indeed, this claim can also be justified with less effort.
Namely, it is not difficult to show that
\begin{equation}
\label{easier}
\# \cR_3(x) \ll    \frac{\# \cQ_3(x)}{(\log x) (\log \log x)^2 }.
\end{equation}
To do so, we note that if $n = pqr \in \cR_3(x)$ then 
$$
p \le (pqr)^{1/3}= n^{1/3} \le x^{1/3}.
$$
Hence, 
$$\# \cR_3(x) \le \sum_{p\le x^{1/3}} \sum_{p< q < 4(p-1)} \min\{\pi(\frac{x}{pq}), \pi(p^2)\}.$$
After making some easy estimates one then arrives at~\eqref{easier}. 

Lemma~\ref{lem:HLLP-3} also motivates us to study  the set $\cE_3$ of exceptional 
ternary integers $n=pqr$, $2<p<q<r$, for which $g(\Psi_n)\ne  2n/p-\deg \Psi_n$. 
We  now
consider a certain subset $\cE_4$ of $\cE_3$ and estimate 
$\# \cE_4(x)$ (Theorem~\ref{E4}). The relevance of this
estimate  becomes clear in the proof of Theorem~\ref{gnoteq} below.

\begin{lemma}
\label{acht}
Let 
$$
\cE_3=\{n=pqr~:~2<p<q<r~{\text{\rm and}}~ g(\Psi_n)\ne  2n/p-\deg \Psi_n. \}.
$$
and
$$
\cE_4=\{n=pqr~:~2<p<q<r~{\text{\rm and}}~ qr<(q+r)(p-1)\}.
$$
We have $\cE_4 \subseteq \cE_3\subseteq \cR_3\subseteq {\mathcal Q}_3$.
\end{lemma}
\begin{proof}
Consider the inequality $p-1>2n/p-\deg\Psi_n$.  As 
$$2n/p-\deg\Psi_n=2qr-pqr+(p-1)(q-1)(r-1),$$ this is equivalent to
\begin{equation}
\label{primefun} 
qr<(q+r)(p-1).
\end{equation}
By Lemma~\ref{lem:HLLP-3}, it now follows that 
$\cE_4 \subseteq \cE_3\subseteq \cR_3$.
The final inclusion is obvious.
\end{proof}
\begin{theorem}
\label{E4}
For the set
$$
\cE_4=\{n=pqr~:~2<p<q<r~{\text{\rm and}}~ qr<(q+r)(p-1)\}
$$
we have
$$
\frac{x}{(\log x)^3} \ll \#\cE_4(x)\ll \frac{x}{(\log x)^3}.
$$
\end{theorem}

\begin{proof} Note that replacing $p>2$ in the definition of $\cE_4$ by $p\ge 2$ does
not change the set.

\subsection*{Upper bound} Let $n\in \cE_4(x)$. We may assume that $n>x/(\log x)^3$, since the set $\{n~:~n\le x/(\log x)^3\}$ has cardinality at most $x/(\log x)^3$.
By Lemma~\ref{acht} we have $\cE_4(x)\subseteq \cR_3(x)$. Thus, writing
$n=pqr$ with $p<q<r$, we have 
$$
4p^4=p (4p) (p^2)>pqr>\frac{x}{(\log x)^3},
$$
therefore $p>x^{1/5}$ for large values of $x$. The inequality defining $\cE_4$ 
is equivalent to 
$$
\(\frac{q}{p-1}\)\(\frac{r}{p-1}\)<\frac{q}{p-1}+\frac{r}{p-1},
$$
or
\begin{equation}
\label{FLeq:1}
\frac{r}{p-1}<\frac{q/(p-1)}{q/(p-1)-1}=\frac{q}{q-p+1}.
\end{equation}
Since $r>q$, we have that $r/(p-1)>q/(p-1)$, giving 
$$
\frac{q}{p-1}<\frac{q/(p-1)}{q/(p-1)-1},
$$
which leads to $q<2(p-1)$. We put $h=q-(p-1)$, so $h<p-1$. Then~\eqref{FLeq:1} is equivalent to
$$
\frac{r}{p-1}<\frac{q}{q-p+1}=\frac{p-1+h}{h}
$$
so
$$
 r<\frac{(p+h-1)(p-1)}{h}.
$$
We also have $r<x/pq$, so we get
\begin{equation}
\label{eq:Bound r}
r<\min\left\{\frac{x}{p(p+h-1)}, \frac{(p+h-1)(p-1)}{h}\right\}.
\end{equation}

We consider both possibilities that may occur on the right hand side of~\eqref{eq:Bound r}
separately. 

\subsubsection*{Case 1} Assume that 
$$
\frac{x}{p(p+h-1)}<\frac{(p+h-1)(p-1)}{h}.
$$
Then 
$$
xh<p(p-1)(p+h-1)^2<p^2 q^2<4 p^4,
$$
so $p>2^{-1/2}(xh)^{1/4}.$ Since $r>p>x^{1/5}$, the number of primes $r<x/pq$ is 
\begin{equation}
\label{FLeq:2}
\pi\(\frac{x}{pq}\) = O\(\frac{x}{pq (\log x)}\).
\end{equation}
We now fix $h$ and sum up the above bound over $p\in [2^{-1/2} (xh)^{1/4}, x^{1/3}]$ such that $p$ and $p+h-1$ are both primes. Let $\sigma(k)$ denote the sum of all integer positive divisors of $k\ge 1$.
By \cite[Theorem 7.3]{MN}, or alternatively by the fundamental lemma of the
sieve applied to the sequence $n(n+h-1)$, see~\cite[Corollary~2.4.1]{HR}, we see 
that the counting function of 
$$
{\mathcal P}_h=\{p: p,~p+h-1~{\text{\rm are~both~primes}}\}
$$
satisfies 
$$
\#\left({\mathcal P}_h\cap [1,t]\right))\ll \frac{t}{(\log t)^2}\prod_{p|(h-1)}\left(1+\frac{1}{p}\right)\ll
\frac{t}{(\log t)^2} \left(\frac{\sigma(h-1)}{h-1}\right),
$$
where we used the observation that
$$\prod_{p|(h-1)}\left(1+\frac{1}{p}\right)\le \sum_{d|(h-1)}\frac{1}{d}=\frac{\sigma(h-1)}{h-1}.$$
With the Abel summation formula, we derive that
\begin{equation*}
\begin{split}
\sum_{\substack{p>2^{-1/2} (xh)^{1/4}\\{p\in \cP_h}}} \frac{1}{pq} &  \le  \sum_{\substack{p>2^{-1/2} (xh)^{1/4}\\ p\in \cP_h}} \frac{1}{p^2} \\
& \ll \frac{1}{(\log x)^2} \(\frac{\sigma(h-1)}{h-1}\)\int_{2^{-1/2} (xh)^{1/4}}^{x^{1/3}} \frac{dt}{t^2}
 \end{split}
\end{equation*}
and hence
\begin{equation}
\label{FLeq:3}
\sum_{\substack{p>2^{-1/2} (xh)^{1/4}\\ p\in \cP_h}} \frac{1}{pq} 
\ll  \frac{1}{x^{1/4} (\log x)^2} \(\frac{\sigma(h-1)}{h-1}\) \frac{1}{h^{1/4}}.
\end{equation}
Thus, on combining~\eqref{FLeq:2} and~\eqref{FLeq:3}, we obtain that, for fixed $h$, the 
number of $n=pqr\in \cE_4(x)$ in  Case~1 is bounded by
\begin{equation}
\label{FLeq:4}
\sum_{\substack{p>2^{-1/2} (xh)^{1/4}\\ p\in \cP_h}}\pi\( \frac{x}{pq}\)=O\(\frac{x^{3/4}}{(\log x)^3} \(\frac{\sigma(h-1)}{h-1}\) \frac{1}{h^{1/4}}\).
\end{equation}
Note that, by the way it was defined, $h$ is odd and $h\ge 3.$ We now sum up over such $h$ and use 
that
$$\sum_{n\le x}\frac{\sigma(n)}{n}=\sum_{n\le x}\sum_{d|n}\frac{1}{d}\le x\sum_{d\le x}\frac{1}{d^2},$$
to see that
$$
\sum_{\substack{3\le h\le t\\ h~{\text{\rm odd}}}} \frac{\sigma(h-1)}{h-1}=O(t).
$$
We use this estimate and 
the Abel summation formula with the sequence $a_h=\sigma(h-1)/(h-1)$ for $h\ge 3$ and odd (and $a_h=0$ otherwise) and the function $f(t)=t^{-1/4}$, to get that
\begin{equation}
\label{FLeq:6}
\sum_{\substack{3\le h\le x^{1/3}\\ h~{\text{\rm odd}}}} \(\frac{\sigma(h-1)}{h-1}\)\frac{1}{h^{1/4}} \ll \int_{1}^{x^{1/3}} \frac{dt}{t^{1/4}}\ll t^{3/4}\Big|_{t=1}^{t=x^{1/3}} \ll x^{1/4},
\end{equation}
which, together with~\eqref{FLeq:4}, gives a bound of 
$O(x/(\log x)^{3})$ for the number of $n\in \cE_4(x)$ in Case 1.

\subsubsection*{Case 2}  Assume that 
$$
\frac{x}{p(p+h-1)}\ge \frac{(p+h-1)(p-1)}{h}.
$$
Then
$$
xh>p(p-1)q^2\ge q^4/4,
$$
giving $q<2^{1/2} (xh)^{1/4}$. In this case, 
$$
r\le \frac{q(p-1)}{h}<\frac{4p^2}{h},
$$
and the number of such primes is 
$$
\pi\(\frac{4p^2}{h}\) = O\(\frac{p^2}{h(\log x)}\),
$$
where we used the observation that $p^2/h>p>x^{1/5}$.
Again, fixing $h$ and summing up over $p,$ we get a bound of
\begin{equation}
\label{FLeq:5}
O\(\frac{1}{h(\log x)} \sum_{\substack{x^{1/5}<p<2^{1/2} (xh)^{1/4}\\ p\in \cP_h}} p^2\).
\end{equation}
Since $p\in \cP_h$, it follows, by the Abel summation formula, that
\begin{eqnarray*}
\sum_{\substack{x^{1/5}\le p<2^{1/2} (xh)^{1/4}\\ p\in \cP_h}} p^2 & \ll & \frac{1}{(\log x)^2} \(\frac{\sigma(h-1)}{h-1}\)\int_{x^{1/5}}^{2^{1/2}(xh)^{1/4}} t^2 dt\\
& \ll & 
\frac{(xh)^{3/4}}{(\log x)^2} \(\frac{\sigma(h-1)}{h-1}\).
\end{eqnarray*}
Inserting this into~\eqref{eq:5}, we conclude that, for fixed $h,$ the number of $n\in \cE_4(x)$ in
Case~2 is 
$$
O\(\frac{x^{3/4}}{(\log x)^3} \(\frac{\sigma(h-1)}{h-1}\) \frac{1}{h^{1/4}}\).
$$
We now sum again over $h< x^{1/3}$ and use~\eqref{FLeq:6} to get that 
the number of $n\in \cE_4(x)$ in Case 2 is $O(x/(\log x)^{3})$.

\subsection*{Lower bound} Let $\varepsilon>0$ to be chosen later. Now suppose
that there are primes $p,q,r$ such that 
\begin{equation}\label{pepsi}
q\in (p,p(1+\varepsilon)),~~~r\in (q,q(1+\varepsilon)).
\end{equation}
Since
$$qr<pq(1+\varepsilon)^2\qquad {\text{\rm and}}\qquad (p+q)(p-1)<(q+r)(p-1),$$ the 
inequality~\eqref{primefun}  holds if 
\begin{equation}\label{ezpq}
1+\frac{p}{q}-\frac{1}{p}-\frac 1q>1 + \frac{1}{1+\epsilon}-\frac{1}{p}-
\frac{1}{q}>(1+\varepsilon)^2.
\end{equation}
Now we choose any $\varepsilon>0$ satisfying $1+1/(1+\varepsilon)>(1+\varepsilon)^3$ and observe
that there exists $p_0(\varepsilon)$ such that if $p\ge p_0(\varepsilon)$, then 
inequality~\eqref{ezpq} is satisfied.

Note that if $p^3(1+\varepsilon)^2\le x$, then for any choice of
$p\ge p_0(\varepsilon)$ and $q,r$ satisfying~\eqref{pepsi}, the inequality~\eqref{primefun} 
is satisfied and $pqr\le x$. 
Putting $c(\varepsilon)=(1+\varepsilon)^{-2/3}$ and noting that by~\eqref{simplepi}, we have 
$$
\pi(x(1+\varepsilon))-\pi(x)= \frac{\(\varepsilon +o(1)\) x}{\log x},
$$ it follows that 
\begin{equation*}
\begin{split}
\# \cE_3(x) & \ge \sum_{p_0(\varepsilon)<p\le c(\varepsilon) x^{1/3}}\sum_{p<q<p(1+\varepsilon)}\sum_{q<r<q(1+\varepsilon)}1\\
& \gg_{\varepsilon}
 \sum_{p_0(\varepsilon)<p\le c(\varepsilon) x^{1/3}} \sum_{p<q<p(1+\varepsilon)} \frac{q}{\log q}\\
 &\gg_{\varepsilon}  \sum_{p_0(\varepsilon)<p\le c(\varepsilon) x^{1/3}}\frac{p^2}{(\log p)^2}
\gg_{\varepsilon} \int_{p_0(\varepsilon)}^{c(\varepsilon) x^{1/3}}\frac{t^2dt}{(\log t)^3}\gg_{\varepsilon} \frac{x}{(\log x)^3}.
 \end{split}
\end{equation*}
Taking, e.g., $\varepsilon=10^{-1}$ completes the proof.
\end{proof}

We have now the ingredients to estimate $\# \cE_3(x)$.

\begin{theorem}
\label{gnoteq}
For the set 
$$
\cE_3=\{n=pqr: 2<p<q<r~{\text{\rm and}}~ g(\Psi_n)\ne  2n/p-\deg \Psi_n. \}.
$$
we have
$$\frac{\# \cQ_3(x)}{(\log x)^2 (\log \log x)^2 }\ll 
\#\cE_3(x)  \ll \frac{\# \cQ_3(x)}{(\log x) (\log \log x)^2 }.
$$
\end{theorem}

\begin{proof} 
The upper bound follows from Lemma~\ref{acht} and~\eqref{easier}. 
The lower bound follows from Lemma~\ref{acht}, Theorem~\ref{E4} and
the asymptotic formula~\eqref{asgelijk} for $\# \cQ_3(x)$.
\end{proof}

\section{Cyclotomic polynomials and numerical semigroups} 
\label{numericalsemi}

\subsection{Numerical semigroups} 

A {\it{numerical semigroup}} $S$ is a submonoid of 
$(\Z_{\ge0},+)$  with finite complement in $\Z_{\ge0}$. The nonnegative integers not in $S$ are called {\it{gaps}}, and the largest gap is the {\it{Frobenius number}}, denoted $ F(S)$. A numerical semigroup admits a unique and finite minimal system of generators; its cardinality is called {\it{embedding dimension}}, denoted $e(S)$, and its elements {\it minimal generators\/}. 
If $S=\langle a_1,\ldots,a_e \rangle$ is a minimally generated monoid (with $a_i$ positive integers), then $S$ is a numerical semigroup if and only if $\gcd(a_1,\ldots,a_e)=1$ (for a comprehensive introduction to numerical semigroups, 
see,~for example,~\cite{roos}). 

To a numerical semigroup $S$ we associate
$H_S(Z)=\sum_{s\in S}Z^s$, its {\it {Hilbert series}}, and $ P_S(Z)=(1-Z)H_S(Z)$, its {\it semigroup polynomial}. (Note that $P_S(Z)$ is indeed a polynomial, of degree $F(S)+1$, since all elements larger than $F(S)$ are in $S$; for the same reason, $H_S(Z)$ is not a polynomial.) 

Solely by the definition of $H_S$ and $P_S$, the
 following is 
easy to establish, see~\cite{Cio}:

\begin{lemma}
\label{coeff-PS}
Let $S$ be a numerical semigroup and assume that $ P_S(Z)=a_0+a_1Z+\cdots+ a_kZ^k$. Then for $s\in \{0,\ldots, k\}$
\[
a_s=\begin{cases}
1 & \hbox{if } s\in S\hbox{ and }s-1\not\in S,\\
-1 & \hbox{if } s\not\in S\hbox{ and }s-1\in S,\\
0 &\hbox{otherwise}.
\end{cases}
\]
\end{lemma}
It is a folklore result (see, for instance~\cite{Mor2}) that, if $S=\langle a,b \rangle$ 
is a numerical semigroup (that is, $\gcd(a,b)=1$), then 
\begin{equation}
\label{pq}
P_S(Z)=\frac{(1-Z)(1-Z^{ab})}{(1-Z^a)(1-Z^b)}.
\end{equation}
Suppose that $S=\langle p,q\rangle$ with $p,q$ distinct primes. Then by~\eqref{prod1}
\begin{equation}
\label{connection}
P_S(Z)=\Phi_{pq}(Z).
\end{equation} 

\subsection{Applications to $\Phi_{pq}$}
 
The identity~\eqref{connection} and Lemma~\ref{coeff-PS} allow one to prove results concerning the coefficients of $\Phi_{pq}$ by studying certain properties of the numerical semigroup $\langle p,q\rangle$. 

Before going further, we give some definitions. 
Let $S$ be a numerical semigroup. 

We say that $D=\{n+1,\ldots, n+k\}\subset \Z_{\ge0}$ is a $k$-{\it gapblock} (or a \textit{gapblock} of size $k$) if $S\cap (D\cup \{n,n+k+1\})=\{n,n+k+1\}$. (In the literature this is also named a $k$-\textit{desert}). 

Similarly, $E=\{n+1,\ldots, n+k\}\subset \Z_{\ge0}$ is a $k$-{\it elementblock} 
(or an \textit{elementblock} of size $k$) if $S\cap (E\cup\{n,n+k+1\})=E$.    

A key fact is that a numerical semigroup of embedding dimension $2$, hence $S=\langle p,q \rangle$ in particular, 
is \textit{symmetric}, that is, $S\cup(F(S)-S)=\Z$, where $F(S)-S=\left\lbrace F(S)-s~:~s\in S\right\rbrace $ (see, for example,~\cite{Mor2} 
and~\cite{roos}). 
This implies that there is a one to one correspondence between
$k$-gapblock and $k$-elementblocks in $S$, which is relevant for our later purposes.

In what follows, the notation $\{0\}_m$ is used to indicate a string
$\underbrace{0,\ldots,0}_{m}$ of $m$ consecutive zeros.

\begin{theorem} 
\label{oana} 
Let $p<q$ be primes. Then
\begin{itemize}
\item[(i)] $g\(\Phi_{pq}\)=p-1$ and the number of maximum gaps equals $2\left\lfloor q/p \right\rfloor$;

\item[(ii)] $\Phi_{pq}(Z)$ contains the sequence of consecutive coefficients of the form $\pm 1,
\{0\}_{m},\mp 1$ for all $m=0,1,\ldots,p-2$ if and only if 
$q\equiv \pm 1 \pmod p$.
\end{itemize}
\end{theorem}
\begin{proof} We prove parts (i) and~(ii) separately.

\subsubsection*{\normalfont(i)} The fact that $g\(\Phi_{pq}\)=p-1$  was already proven in~\cite{Mor2}, so let us only deal with the second claim. Set $q=pk+h$, with $k=\left\lfloor q/p \right\rfloor$ and $1\le h\le p-1$. Consider $S=\langle p,q \rangle$ and note that, by Lemma~\ref{coeff-PS}, a maximum gap in $\Phi_{pq}$ corresponds to both a $(p-1)$-gapblock and a $(p-1)$-elementblock in $S$. Let us see what the elements of $S$ are. First, we have $kp\in S$, for all $k\in\Z_{\ge0}$. Next, the smallest element in $S$ is $q=pk+h$, which lies in the interval $(pk,p(k+1))$. It is clear now that we have precisely $\left\lfloor q/p \right\rfloor$ gapblocks of size $(p-1)$ in $S$, namely $\left\lbrace jp+1,\ldots,jp+p-1\right\rbrace $, for $j=0,1,\ldots, k-1$. Since $S$ is symmetric, to any $k$-gapblock in there corresponds a $k$-elementblock and conversely. Thus, there are $\left\lfloor q/p \right\rfloor$ gapblocks and $\left\lfloor q/p \right\rfloor$ elementblocks of size $p-1$ in $S$, leading to $2\fl{q/p}$ maximum gaps 
in $\Phi_{pq}$.   

\subsubsection*{\normalfont(ii)} Let $S=\langle p,q\rangle$. 
By~\eqref{connection}, we have $P_S=\Phi_{pq}$. Note that, in light of Lemma~\ref{coeff-PS}, the sequence $1,\{0\}_{m},-1$ of coefficients of $\Phi_{pq}$ corresponds to an $(m+1)$-gapblock in $S$ (similarly, the sequence $-1,\{0\}_{m},1$ corresponds to an $(m+1)$-elementblock in $S $) and conversely. 
Since $\Phi_{pq}$ is self-reciprocal or equivalently, $S=\langle p,q \rangle$ is symmetric, 
it is enough to consider only the sequences $1,\{0\}_{m},-1$.
Therefore, proving part~(ii)
 is equivalent to showing the following:
 \begin{center}\it
$S=\langle p,q \rangle$ has gapblocks of sizes $1,2,\ldots,p-1$ 
if and only if $q\equiv \pm 1\pmod p$.
\end{center}
To prove this claim, we consider both implications separately.
First, assume $q\equiv \pm 1\pmod p$. We only treat the case $q\equiv 1\pmod p$, as the other can be dealt with in a similar manner.
Let $q=pk+1$ for some $k\in\Z_{>0}$. Then, for $1\leq m\leq p-1$, the intervals $\cI_m=[mpk,\ldots,mpk+p)$ are pairwise disjoint. Next, note that if $a,b\in\Z_{\ge0}$ are so that $mpk\leq a p+b q<mpk+p$, then $b\leq m$. Conversely, for any such $b$ there is a unique $a\in\Z_{\ge0}$ such that $mpk\leq a p+b q<mpk+p$, since for a fixed $0\leq b\leq m$, exactly one number from $\left \{ a p+b q~:~a\in\Z_{\ge0}\right\}$  lands in $\cI_m$. 
We can write any number 
$$
mpk +h = (m-h)kp + hq, \quad\text{for~} h = 0,1, \ldots, m,
$$
in the form $a p+b q$, with $0\leq b\leq m$
and, by the above, no other element of $\cI_m$ can be written in this way. Therefore  $$
\cI_m\cap S=[mpk,mpk+1,\ldots,mpk+m]
$$ and $\left\lbrace mpk+m+1,\ldots,mpk+p-1\right\rbrace $ is a $(p-m)$-gapblock of $S$, for $m=1,2,\ldots,p-1$. 

By considering the intervals $\cI_m=(mpk-p,\ldots,mpk]$ we can give a similar argument in case $q\equiv-1\pmod p$, on taking $k=(q+1)/p$.

Conversely, since the intervals $\cJ_k=[pk,\ldots,p(k+1))$ where $k\in\Z_{\ge0}$, form a partition of $\Z_{\ge0}$, $q$ must lie in some $\cJ_k,$ $k\ge 1$ (as $p<q$) and $q\neq pk$. We claim that  
\begin{equation} \label{gaps}q=pk+1\quad \text{or}\quad q=p(k+1)-1.\end{equation} 
Assume otherwise. Note that $\cJ_k\cap S=\left\{pk,q \right\}$, because if $s=a p+b q\in \cJ_k\cap S$ and $s\neq pk$, then $b\geq 1$. If $b\geq 2$, then $s>p+q>p(k+1)$, and $s\notin \cJ_k$. Hence, $b=1$. Next, $s\neq q$ implies $a\geq 1$ and then again $s\geq p+q>p(k+1)$. Therefore we obtain two gapblocks $\left\lbrace pk+1,\ldots,q-1\right\rbrace $ and $\left\lbrace q+1,\ldots,p(k+1)-1\right\rbrace $, of size at most $p-3$. But now, as $\cJ_h\cap S=\left\lbrace ph,p(h+1)-1\right\rbrace$, for all $0\leq h\leq k-1$, we have only $(p-1)$-gapblocks before $\cJ_k$. Also, $q+lp\in \cJ_{k+l}\cap S$ is different from $p(k+l)+1$ and $p(k+l+1)-1$, hence there can be no $(p-2)$-gapblock in $\cJ_{k+l}$, for $l\geq 1$. But then $S$ has no $(p-2)$-gapblock, contradiction. Hence,~\eqref{gaps} is true and yields $q \equiv\pm 1\pmod p$. 
\end{proof}

Let $N_2(x)$ be the number of cyclotomic polynomials $\Phi_{n}$ with $n = pq\le x$ and 
$q\equiv \pm 1\pmod p$. By the Dirichlet Theorem on primes in arithmetic progressions, we see that 
$N_2(x)\rightarrow \infty$ as $x\rightarrow \infty$.
We can make this more precise.

\begin{cor}\label{coroana}
We have
$$
N_2(x)= C \frac{x}{\log x}+O\(\frac{x}{(\log x)^2}\),
$$
where
$$
C = \frac{1}{2} + \sum_{p \ge  3}  \frac{2}{p(p-1)}=-\frac{1}{2}+2\sum_{k=2}^{\infty}\log \zeta(k)\frac{(\varphi(k)-\mu(k))}{k}.$$
Numerically
$$C=1.0463133380995902557287349197118847 \ldots.$$
\end{cor}

\begin{proof} For a real $z \ge 1$ and integers $k > a\ge 1$, let, as usual, $\pi(z, k, a)$ 
denote the number of primes $q \le z$ in the arithmetic
progression $q \equiv a \pmod k$.
We have
\begin{equation}
\label{eq:N2}
N_2(x) = \pi\( \frac{x}{2}\)  + \sum_{3 \le p \le x} \( \pi\( \frac{x}{p},p,1\ \)  + 
\pi\( \frac{x}{p},p,-1\)  \) .
\end{equation}
We now choose 
$$
y= \log x \mand Y = x^{1/3}
$$ 
and for $p< y$, we use the  Siegel-Walfisz theorem
(see~\cite[Corollary~5.29]{IwKow})  in the form
$$
\pi(z, k, a) = \frac{\pi(z)}{\varphi(k)} + O\( \frac{z}{\log z)^{3}}\).
$$
For primes $p \in [y,Y]$, we use the Brun-Titchmarsh  theorem 
(see~\cite[Theorem~6.6]{IwKow}) in the form 
$$
\pi(z, k, a) \ll \frac{z}{\varphi(k) \log(2z/k)}, 
$$
which holds for any positive integer $k < z$. 
Finally, for $p > Y$, we use the trivial bound 
$$
\pi(z, k, a) \ll \frac{z}{k} .
$$
So, collecting the above estimate and bounds we obtain
\begin{equation*}
\begin{split}
N_2(x)   &=    \pi(x/2) + 2\sum_{3 \le p < y}   \( \frac{\pi(x/p)}{p-1} + 
O\(\frac{x}{p(\log x)^{3}}\)\)\\
& \phantom{{}=\pi(x/2)}+  \sum_{y \le p \le Y} \frac{x}{p^2 \log x}  +  
\sum_{Y < p \le x} \frac{x}{p^2}, 
\end{split}
\end{equation*}
which yields
\begin{equation}
\label{eq:N2 expand}
N_2(x) =    \pi(x/2) +2 \sum_{3 \le p < y}  \frac{\pi(x/p)}{p-1} + O\(\frac{x}{(\log x)^{2}}\).
\end{equation}
Now, using~\eqref{simplepi}, we derive
\begin{align*}
\pi(x/p) & =  \frac{x}{p(\log x - \log p)} + O\(\frac{x}{p(\log x)^{2}}\) \\
& =  \frac{x}{p \log x} + O\(\frac{x\log p}{p(\log x)^{2}}\). 
\end{align*}
Hence, 
\begin{align*}
 \sum_{3 \le p < y}  \frac{\pi(x/p)}{p-1}  
& =  \frac{x}{\log x}  \sum_{3 \le p < y}  \( \frac{1}{p(p-1)} + O\(\frac{\log p}{p^2 \log x}\)\) \\
& =  \frac{x}{\log x}  \sum_{p\ge 3}  \frac{1}{p(p-1)} + O\(\frac{x}{(\log x)^2}\(\sum_{p\ge 3} \frac{\log p}{p^2}\)\) \\
& =  \frac{x}{\log x}  \sum_{p \ge  3}  \frac{1}{p(p-1)} +   O\(\frac{x}{(\log x)^{2}}\).
\end{align*}
Substituting this bound into~\eqref{eq:N2 expand}, we obtain
the desired result with $C$ equal to the prime sum given. This sum is not suitable
for obtaining $C$ with high numerical precision. In order to 
achieve that, 
we express $C$ in terms of values of the zeta-values at integer arguments (which
can be approximated with enormous numerical precision very fast).
Using the well-known formulas
$$\sum_{p\ge 2} \frac{1}{p^s}=\sum_{n=1}^{\infty}\frac{\mu(n)}{n}\log \zeta(ns){\rm ~and~}
\frac{\varphi(n)}{n}=\sum_{d|n}\frac{\mu(d)}{d},$$ 
we obtain
\begin{equation*}
\begin{split}
\sum_{p\ge 2} \frac{1}{p(p-1)} & =\sum_{m=2}^{\infty}\sum_{p\ge 2}\frac{1}{p^k}=
\sum_{m=2}^{\infty}\sum_{n=1}^{\infty}\frac{\mu(n)}{n}\log \zeta(mn)\\
& = \sum_{k=2}^{\infty}\log \zeta(k) \sum_{n\cdot m=k,~m\ge 2}
\frac{\mu(n)}{n}\\
& =\sum_{k=2}^{\infty}\log \zeta(k)\frac{(\varphi(k)-\mu(k))}{k}.
\end{split}
\end{equation*}
On noting that the prime sum above equals $C/2+1/4$ the proof of the second equality for
$C$ is completed. This identity allows one to evaluate $C$ with hundreds of decimals of
precision.
\end{proof}
\begin{rem}
If we are interested in the number of binary integers
$n = pq\le x$ with
$q\equiv \pm 1\pmod p$, then the above estimate holds with
$C$ replaced by $C-1/2$. 
\end{rem}

\begin{rem} 
For more details on expressing prime products and 
prime sums (such as $C$) in terms of
zeta-values, see, for example,~\cite[pp.~208-210]{Coh} and~\cite{MazPet,Mor0}.
\end{rem}

\subsection{A small generalization} 
\label{smallgen}
The attentive reader might have noticed that in our proofs of 
Theorem~\ref{oana} we did not
use that $p$ and $q$ are prime, but only that they are coprime. We   
explore now what happens if we consider  $S=\langle a,b \rangle$ with
$2\le a<b$ and $a$ and $b$ coprime. The coprime condition ensures that $S$ is a numerical
semigroup. Our earlier proofs then apply to
$$
P_S(Z)=\frac{(1-Z)(1-Z^{ab})}{(1-Z^a)(1-Z^b)}=\prod_{d|ab,d\nmid a, d\nmid b}\Phi_d(Z),
$$
where we used~\eqref{pq} and~\eqref{allebegin...}.
The polynomial $P_{\langle a,b\rangle}(Z)$ is a {\it binary inclusion-exclusion 
polynomial}, see~\cite{bach}.
\begin{theorem}  
\label{maxgenbin}
Let $2\le a<b$ be coprime positive integers. Then 
\begin{itemize}
\item[(i)] the maximum gap
in $\prod_{d|ab,~d\nmid a,d\nmid b}\Phi_d(Z)$ equals $a-1$ and it occurs 
precisely $2\fl{b/a}$ times;

\item[(ii)] the polynomial in {\normalfont(i)} contains the sequence of consecutive coefficients $\pm 1,\{0\}_{m},\mp 1$ for all $m=0,1,\ldots,a-2$ if and only if $b\equiv \pm 1 \pmod a$.
\end{itemize}
\end{theorem}

The result that the 
maximum gap in $\prod_{d|ab,~d\nmid a,d\nmid b}\Phi_d(Z)$ equals $a-1$ already
appears in Moree~\cite{Mor2}.

\section{Bound on jumps}

\subsection{Definitions and background}
\label{backg}
It has been shown by Gallot and  Moree \cite{GaMo1}
that consecutive coefficients of ternary cyclotomic polynomials differ
by at most one. Bzd{\c e}ga~\cite{Bzd-3} gave a very different reproof of
this result and initiated the study of 
the number of ``jumps''. More precisely, since the sequence of 
the  coefficients
$a_n(k)$ is symmetric~\eqref{eq:sym},
the number of jumps up with $a_n(k) = a_n(k-1) + 1$
is the same as the number of jumps down with $a_n(k) = a_n(k-1) - 1$.
We denote this common number by $J_n$. 

Bzd{\c e}ga~\cite{Bzd-3} has shown that 
$$J_n > n^{1/3}$$ 
for any ternary cyclotomic polynomial
$\Phi_n$ and  conditionally, on a certain variant of
the prime $3$-tuple conjecture, that
\begin{equation}
\label{jeejee}
J_n < 15 n^{1/3}
\end{equation} 
for infinitely many $n$. 
As a special case he showed that if $m, 6m-1$ and $12m-1$ are all prime for some $m\ge 7$, 
then~\eqref{jeejee} holds with $n=m(6m-1)(12m-1)$.
Here, we   prove an unconditional variant of the upper bound~\eqref{jeejee}. 
\begin{theorem} \label{thm: Jn} 
For infinitely many $n=pqr$ with 
pairwise distinct odd primes $p$, $q$ and $r$, we have
$$
J_n \ll n^{7/8+ o(1)}.
$$
\end{theorem} 

\subsection{Preparations}

For two integers $k$ and $m$ with $\gcd(k,m)=1$ we denote by
$k_m^*$ the unique integer defined by the conditions
$$
k k_m^*\equiv 1 \pmod m \mand 1 \le k_m^* < m.
$$
Our approach depends on showing that for almost all primes $p$ 
there are many primes $\ell$ that are not too large
and such that $\ell_p^*$ belongs to 
a certain interval. 

It is natural that we study this problem using bounds of exponential sums
involving reciprocals of primes~\cite{Bak,FouShp,Irv}.

For a prime $p$, an integer $a$ and a real number $L$ we 
define (recall that $\ell$ and $p$ always denote prime numbers) the double exponential sum
$$
S(a;L,P) = \sum_{\ell \sim L} \sum_{p \sim P}
\exp(2 \pi i a \ell_p^*/p), 
$$
and formulate a special case of a result of 
Duke,  Friedlander
and Iwaniec~\cite[Theorem~1]{DuFrIw2},

\begin{lemma}
\label{lem:Irv} Let $L$ and $P$ be arbitrary real positive numbers  
and $a$ an 
integer satisfying $1 \le |a| \le LP$. Then
$$
|S(a;L,P)|
\le \(L^{1/2} + P^{1/2} + \min\{L,P\}\) (LP)^{1/2+o(1)} .
$$
\end{lemma}

We now state the  {\it Erd{\H o}s--Tur{\'a}n inequality\/}
(see~\cite{DrTi,KuNi}),  which  links the distributional
properties of sequences with exponential sums. 

To make this precise, for a sequence of $N$ real numbers 
$\Gamma = \(\gamma_{n} \)_{n=1}^N$ 
of the half-open interval $[0,1)$, we denote by
$\Delta_\Gamma$ its {\it discrepancy\/}, that is,
$$
\Delta_\Gamma = \sup_{0 \le \alpha \le 1}
\left|T_\Gamma(\alpha) - \alpha N \right|,
$$
where $T_\Gamma(\alpha)$ is the number of points of the sequence $\Gamma$
in the interval $[0, \alpha]$. 

\begin{lemma}
\label{lem:ErdTur} For any integer $H \ge 1$, the discrepancy $\Delta_\Gamma$ of a sequence 
$\Gamma = \(\gamma_{n} \)_{n=1}^N$ of
$N$ real numbers $\gamma_1,\ldots,\gamma_N\in[0,1)$ 
satisfies the inequality
$$
\Delta_\Gamma  \ll  \frac{N}{H} + \sum_{a=1}^H \frac{1}{a} \left| \sum_{n=1}^N \exp(2 \pi i a \gamma_n)\right|. 
$$
\end{lemma}

Let $D_p(L,P)$ be the discrepancy of the sequence of fractions
$$
\frac{\ell_p^*}{p}, \qquad \ell \sim L,\ p\sim P.
$$

Combining Lemma~\ref{lem:Irv} with Lemma~\ref{lem:ErdTur}
(applied with the parameter $H = L$), we obtain 
\begin{equation*}
\begin{split}
D_p(L,P)&\ll  \frac{(\pi(2L)-\pi(L))(\pi(2P)-\pi(P))}{L}\\
& + \(L^{1/2} + P^{1/2} + \min\{L,P\}\) (LP)^{1/2+o(1)},
 \end{split}
\end{equation*}
and hence the following corollary.
\begin{cor}
\label{cor:Discr} 
Let $\delta>0$ be arbitrary.
For any real positive $L$ and $P$ with $L \ge P^{\delta}$ 
we have
$$
D_p(L,P)
\le   \(L^{1/2} + P^{1/2} + \min\{L,P\}\) (LP)^{1/2+o(1)},
$$
where the implied constant may depend on $\delta$.
\end{cor}

\subsection{Proof of Theorem~\ref{thm: Jn}}

\begin{proof} We always assume that $L \le P$. Thus, the bound of 
Corollary~\ref{cor:Discr} simplifies as 
$$
D_p(L,P)
\le    L^{1/2} P^{1+o(1)} + L^{3/2} P^{1/2+o(1)}.
$$

We see from Corollary~\ref{cor:Discr} 
that for any $L$ and $P$ and positive integer $h\le P$ 
there are 
\begin{equation}
\label{eq:W}
\begin{split}
W(h,P,L) &= h  \frac{(\pi(2L)-\pi(L))(\pi(2P)-\pi(P))}{P} \\
& + O\(  L^{1/2} P^{1+o(1)} + L^{3/2} P^{1/2+o(1)} \)\\
& =(1+o(1))\frac{hL}{\log L \log P}\\
& +  O\(  L^{1/2} P^{1+o(1)} + L^{3/2} P^{1/2+o(1)} \)
 \end{split}
\end{equation}
pairs $(\ell, p)$ of primes $\ell \sim L$, $p \sim P$ with $\ell_p^* \in [1,hp/P]$. 

We now set $L = \rf{P^\rho}$ for some positive constant $\rho < 1$, fix 
some $\varepsilon> 0$ (sufficiently small compared to $\rho$) 
and set
\begin{equation}
\label{eq:h}
h = L^{-1/2} P^{1+\varepsilon}  + L^{1/2} P^{1/2+\varepsilon}.
\end{equation}
In particular, we assume that $\varepsilon< \min\{\rho/2,(1- \rho)/2\}$  so that for  a
sufficiently large $P$ we have $h < P$.

We see from~\eqref{eq:W} that with this choice of $h$ we have
$$
W(h,P,L)  =(1+o(1))\frac{hL}{\log L \log P} > \pi(2P)-\pi(P).
$$
Hence, there is a prime number $p \sim P$ such that for 
at least  two pairs $(q,p)$ and
$(r,p)$ counted by $W(h,L,P)$ we have $q_p^*, r_p^* \in [1,hp/P]$.

We define the quantities 
$\alpha_p(q,r)$ and $\beta_p(q,r)$ 
as the smallest and the second smallest element in the set
$\{q_p^*, p-q_p^*, r_p^*, p-r_p^*\}$. 
Hence, the above choice of $q$ and $r$ implies that 
\begin{equation}
\label{eq:abp}
\alpha_p(q,r),\beta_p(q,r)  \ll h . 
\end{equation}
Now writing 
$$
qq_p^* = 1 +kp
$$
for some integer $k$, we see that 
$$
p_q^* = q-k = q- (qq_p^*-1)/p  = q + O(qh/p) = q + O(Lh/P)
$$
and similarly 
$$
p_r^* =  r + O(Lh/P).
$$
Defining $\alpha_q(p,r)$, $\beta_q(p,r)$, $\alpha_r(p,q)$ and $\beta_r(p,q)$
in full analogy with $\alpha_p(q,r)$ and $\beta_p(q,r)$, we conclude that 
\begin{equation}
\label{eq:abqr}
\alpha_q(p,r), \beta_q(p,r) , \alpha_r(p,q), \beta_r(p,q)  \ll Lh/P.
\end{equation}

We now see from~\eqref{eq:abp} and~\eqref{eq:abqr} that in the notations of the proof of~\cite[Theorem~4.1]{Bzd-3}, 
we have
$$
R   \ll hL^2,\quad 
S \ll L^2h^3/P^2, \quad 
T \ll  L^2h^2/P.
$$
Hence,  we derive
$$
R+S+T \ll  hL^2 + L^2h^3/P^2 +  L^2h^2/P \ll hL^2.
$$
Now, by~\cite[Theorem~4.1]{Bzd-3}, using~\eqref{eq:abp} and~\eqref{eq:abqr}, 
after simple calculations,  we derive
\begin{equation}
\label{eq:Jn}
J_n \ll  h L^2  \ll n (h/P). 
\end{equation}
Recalling that  $L= P^\rho + O(1)$, for a sufficiently large $P$ we obtain 
$n \asymp P^{1+ 2 \rho}$. 
Since $\varepsilon>0$ can be arbitrary small, 
the bound~\eqref{eq:Jn} implies that 
$$
J_n \ll n^{1-\vartheta +o(1)},
$$
where, recalling~\eqref{eq:h}, we have
$
(1+2\rho)\vartheta = \min\{\rho/2, 1/2-\rho/2\}.
$
Choosing $\rho = 1/2$ we obtain $\vartheta= 1/8$ and the
result follows. 
\end{proof}

\section{Recapitulation of the results obtained}
\label{recap}

Here we give a short summary of the results obtained, which are  mainly 
motivated by, and related to, those from~\cite{Bzd-3} and~\cite{HLLP}.

The paper~\cite{HLLP} on maximum gaps in (inverse) cyclotomic polynomials:
\begin{itemize}
\item We provide an asymptotic for the quantity $\# \cR_3(x)$  (Theorem~\ref{florian}), cf. 
Corollary~\ref{precisie}. This asymptotic makes Remark 1 in 
~\cite{HLLP} that $\# \cR_3(x)$ is small quantitative.
\item The right order of magnitude for the number of $n=pqr\le x$ with
 $p-1 > 2n/p-{\rm deg}\Psi_n$ is provided (Theorem~\ref{E4}).
\item Estimates for the number of $n=pqr\le x$ with $g(\Psi_n)\ne 2n/p-{\rm deg}\Psi_n$ are 
provided (Theorem~\ref{gnoteq}).
\item The gap problem for binary cyclotomic polynomials is reformulated in terms of 
numerical semigroups (Section~\ref{numericalsemi}).
\item Using this approach the gap structure is studied more in detail (Theorem~\ref{oana}).
\item An easy reproof of $g(\Phi_{pq})=p-1$ due to Nathan Kaplan is given (Section~\ref{achtergrond}).
\item The $\Phi_{pq}$ with $q\equiv \pm1\pmod p$ are shown to be special 
(Theorem~\ref{oana}) and their number for
$pq\le x$ quantified (Corollary~\ref{coroana}).
\item The gap structure for
binary inclusion-exclusion polynomials is considered (Section~\ref{smallgen}).
\end{itemize}

The paper~\cite{Bzd-3} on jumps of cyclotomic polynomials:
\begin{itemize}
\item We use bounds of double Kloosterman sums over primes to  show that $J_n\ll n^{7/8+o(1)}$ for
infinitely many ternary $n$ (Theorem~\ref{thm: Jn}).
\end{itemize}

\section*{Acknowledgements}
The authors would like to thank Nathan Kaplan for his permission to 
present his elegant proof of~\eqref{eq:pq} (communicated by
e-mail to the third author). Further, Alessandro Languasco and Alessandro
Zaccagnini for e-mail correspondence regarding the numerical evaluation of 
the constant $C$ that appears in Corollary~\ref{coroana}.

The first author worked on cyclotomic coefficient gaps while carrying an internship at the Max Planck Institute for Mathematics in August 2010.
The second author worked on this paper
at the Max Planck Institute for Mathematics
 during a one-week visit in January 2014, 
the third author in April 2015 and the fifth author during a visit 
July-December 2013. All authors gratefully acknowledge the support, the hospitality 
and the excellent conditions for collaboration at the Max Planck Institute for Mathematics. 
The project was continued whilst the authors were working
at other institutions.
The fifth author was also supported in part by the ARC Grants
 DP130100237 and DPDP140100118.

\end{document}